\documentclass{amsart}

\usepackage[nobysame]{amsrefs}
\usepackage{amsthm}
\usepackage{amssymb}

\usepackage{tikz-cd}
\usepackage{colonequals}
\usepackage{enumerate}
\usepackage{eucal}

\usepackage{hyperref}
\hypersetup{%
  bookmarksnumbered=true,%
  colorlinks=true,%
  linkcolor=blue,%
  citecolor=blue,%
  filecolor=blue,%
  menucolor=blue,%
  urlcolor=blue,%
  bookmarksopen=true,%
  bookmarksdepth=2,%
  pageanchor=true}

\usepackage{longtable}

\makeatletter
\@namedef{subjclassname@2020}{\textup{2020} Mathematics Subject Classification}
\makeatother

\numberwithin{equation}{section}

\theoremstyle{plain}
\newcounter{intro}

\newtheorem{theorem}[equation]{Theorem}
\newtheorem{hypothesis}[equation]{Hypothesis}
\newtheorem{proposition}[equation]{Proposition}
\newtheorem{lemma}[equation]{Lemma} 
\newtheorem{corollary}[equation]{Corollary}

\theoremstyle{definition}

\newtheorem{chunk}[equation]{}

\theoremstyle{remark}

\newcommand{\acat}{\operatorname{C}_{\mathcal{O}}}
\newcommand{\con}[1]{\Phi_{#1}}
\newcommand{\Coker}{\operatorname{Coker}}
\newcommand{\cmod}[1]{\Psi_{#1}}
\newcommand{\depth}{\operatorname{depth}}
\newcommand{\eps}{\varepsilon}
\newcommand{\Ext}{\operatorname{Ext}}

\newcommand{\ecoh}{\operatorname{F}}
\newcommand{\height}{\operatorname{height}}
\newcommand{\hh}{\operatorname{H}}
\newcommand{\Ker}{\operatorname{Ker}}
\newcommand{\length}{\operatorname{length}}
\newcommand{\ord}{\operatorname{ord}}
\newcommand{\pos}[1]{[\![{#1}]\!]}
\newcommand{\rank}{\operatorname{rank}}
\newcommand{\Spec}{\operatorname{Spec}}
\newcommand{\tors}{\operatorname{tors}}
\newcommand{\tfree}[1]{{#1}^{\operatorname{tf}}}
\newcommand{\vf}{\varphi}
\newcommand{\fm}{\mathfrak{m}} 
\newcommand{\fp}{\mathfrak{p}}
\newcommand{\lra}{\longrightarrow}
\newcommand{\xra}{\xrightarrow}
\newcommand{\bs}{\boldsymbol}
\newcommand{\mco}{\mathcal{O}}
\newcommand{\mcv}{\mathcal{V}}
\newcommand{\mcz}{\mathcal{Z}}
\newcommand{\rhobar}{\overline{\rho}}

\newcommand{\GL}{\mathrm{GL}}
\newcommand{\PGL}{\mathrm{PGL}}

\DeclareMathOperator{\Frob}{Frob}
\newcommand{\mcn}{\mathcal{N}}
\newcommand{\bbA}{\mathbb{A}}
\newcommand{\TT}{\mathbb{T}}
\newcommand{\QQ}{\mathbb{Q}}

\newcommand{\sm}{\smallsetminus}
\newcommand{\es}{\varnothing}

\newcommand{\onto}{\twoheadrightarrow}

\begin{document}

\title[Freeness of Hecke modules]{Freeness of Hecke modules at non-minimal levels}

\author[S.~B.~Iyengar]{Srikanth B.~Iyengar}
\address{Department of Mathematics,
University of Utah, Salt Lake City, UT 84112, U.S.A.}
\email{iyengar@math.utah.edu}

\author[C.~B.~Khare]{Chandrashekhar  B. Khare}
\address{Department of Mathematics,
University of California, Los Angeles, CA 90095, U.S.A.}
\email{shekhar@math.ucla.edu}

\author[J.~Manning]{Jeffrey Manning}
\address{Max Planck Institute for Mathematics,
Vivatsgasse 7, 53111 Bonn, Germany}
\email{manning@mpim-bonn.mpg.de }

\date{\today}


\begin{abstract} 
 We build on the results of \cite{Iyengar/Khare/Manning:2022a} to show that the homology group $\hh_{r_1+r_2}(Y_0(\mcn_\Sigma),\mco)_{\fm_\Sigma}$ of arithmetic manifolds are free over  certain deformation rings $R_\Sigma$, when there are enough geometric characteristic 0 representations. Hitherto we had proved that the homology group has a nonzero free $R_\Sigma$-direct summand. The new ingredient is a commutative algebra argument involving  congruence modules defined in higher codimension in \cite{Iyengar/Khare/Manning:2022a}.

\end{abstract}

\maketitle

\setcounter{tocdepth}{1}
\tableofcontents

\section{Introduction}
\label{se:intro}
This work concerns the action of Hecke algebras and Galois deformation rings on homology groups of certain arithmetic manifolds. Let $X_0(N)$ be a modular curve, $\mco$ the ring of integers of a finite extension of $\QQ_\ell$, and $\fm$ a non-Eisenstein maximal ideal of the Hecke algebra $\TT$ acting on $\hh_1(X_0(N),\mco)_{\fm}$. Using the modularity lifting results of Wiles~\cite{Wiles:1995}, and Taylor and Wiles~\cite{Taylor/Wiles:1995}, Diamond~\cite{Diamond:1997} proved  that  a certain deformation ring  $R$ acts freely  on   $\hh_1(X_0(N),\mco)_{\fm}$, so the map $R \onto \TT$, through which the $R$ action on Hecke module factors, is an isomorphism. In particular, the  $\TT$-module  $\hh_1(X_0(N),\mco)_{\fm}$ is free. Freeness  of this Hecke module was proved earlier by Mazur \cite{Mazur} using geometric arguments, and even without the assumption that $\fm$ is non-Eisenstein.

 Diamond uses patching  to deduce freeness of the Hecke module over $R$  when $N$ is the minimal level for the residual representation $\rhobar_{\fm}$ attached to $\fm$. For non-minimal levels $N$, he  uses a numerical criterion, \cite[Theorem 2.4]{Diamond:1997}, for a finitely generated $R$-module  to be free and for $R$ to be a complete intersection. The fact that  the criterion is applicable  is verified by level raising arguments a la Ribet, that had already been used in similar contexts in \cite{Wiles:1995}. Use of  classical multiplicity one results for modular forms  which prove such freeness results after  tensoring with the fraction field  of $\mco$ were another important ingredient in being able to apply the numerical criterion.

In   \cite{Iyengar/Khare/Manning:2022a} we developed a higher codimensional  version of the Wiles--Lenstra--Diamond numerical criterion  for finitely generated modules $M$  of sufficient depth over certain families of rings $R$ to  have a free direct summand of positive rank, and for $R$ to be complete intersection. We applied this to the ring $R_\infty$ and module $M_\infty$ produced by the patching method to deduce that $M_\infty$ has a free direct summand of some positive rank. From this it follows that the lowest degree non-vanishing homology group $\hh_d(Y,\mco)_{\fm}$ of a certain arithmetic manifold $Y$ associated to $\PGL_2$ over an arbitrary number field $F$ also has a free direct summand of positive rank as a module over the Hecke algebra $\TT$.

On the other hand our methods did not yield freeness of the Hecke module, in contrast to the result of \cite{Diamond:1997}. The reason for this is that patching gives no direct control over the generic rank of the module $M_\infty$. Typically the only way of determining this generic rank is to use classical multiplicity one theorems for automorphic forms to determine the rank of $M_\infty$ at all $\overline{\QQ}_\ell$-points of $\Spec \TT\subseteq \Spec R_\infty$. This approach  only works if $\Spec \TT$ has enough $\overline{\QQ}_\ell$-points --- to directly deduce freeness from the main result of \cite{Iyengar/Khare/Manning:2022a} one would need to have a $\overline{\QQ}_\ell$-point of $\Spec \TT$ on each irreducible component of $\Spec R_\infty$. This issue also arises in \cite[Theorem 1.3]{Calegari/Geraghty:2018}.

In the modular curve case considered by Diamond, this is not an issue as the Hecke algebra is flat over $\mco$ and so  has enough $\overline{\QQ}_\ell$-points. For arithmetic manifolds associated to arbitrary number fields, the Hecke algebra $\TT$ is, in general, not flat over $\mco$ and may not have enough $\overline{\QQ}_\ell$-points to determine the generic rank of $M_\infty$. In fact, $\TT$ may well have no $\overline{\QQ}_\ell$-points at all! Thus the requirement that $\Spec \TT$ have a $\overline{\QQ}_\ell$-point on each irreducible component of $\Spec R_\infty$ is extremely restrictive.

In this paper we develop new commutative algebra arguments, building on those of  \cite{Iyengar/Khare/Manning:2022a}, to show freeness of patched modules under far less restrictive conditions. The main result is Theorem \ref{th:only} which allows us to use information about the rank of $M_\infty$ over one component of $\Spec R_\infty$ to determine its rank over the other components. 

When the number field $F$ is totally complex (considered in Theorem \ref{th:mult 1}), this allows us to deduce freeness assuming only the existence of a $\overline{\QQ}_\ell$-point of $\Spec \TT$ lying on one particular component of $\Spec R_\infty$ (the ``highest level'' component). This is equivalent to assuming the existence of a particular geometric characteristic 0 lift of $\rhobar_{\fm}$ (ramified at a specified set of primes $\Sigma$, as well as some other auxiliary primes). One might expect that such  a lift  always exists (although this is far from known in general), and so our work represents a plausible approach for proving freeness of Hecke modules associated to totally complex fields.

For an arbitrary number field $F$ (considered in Theorem \ref{th:mult 2^r1}) the situation is somewhat more complicated. In order to deduce freeness from our method, one needs to prove an appropriate lower bound on the generic rank of $M_\infty$ on the ``minimal level'' component in addition to having a characteristic $0$ lift of $\rhobar_{\fm}$ corresponding to a point on the ``highest level'' component. In the case when $F$ is totally complex, the required lower bound is $1$ and so one gets this for free. However in the case of a general $F$, we are only able to prove this in the case when $\Spec \TT$ has an additional $\overline{\QQ}_\ell$-point on the ``minimal level'' component of $\Spec R_\infty$.

\section{Congruence modules and Wiles defect}\label{se:congruence}
In this section we recall the construction and basic properties of congruence modules and Wiles defect, introduced in \cite[Section 2]{Iyengar/Khare/Manning:2022a}, for modules over  local rings. 

\begin{chunk}
	Let $\mco$ be a discrete valuation ring, with valuation $\ord(-)$ and uniformizer $\varpi$. Throughout we fix a complete local $\mco$-algebra $A$ and a finitely generated $A$-module $M$. Given a map $\lambda\colon A\to \mco$ of $\mco$-algebras, set $\fp_\lambda \colonequals \Ker(\lambda)$ and 
	\[
	\con{\lambda}(A)\colonequals \tors(\fp_\lambda/\fp_\lambda^2)\,,
	\]
	namely, the torsion part of the cotangent module $\fp_\lambda/\fp_\lambda^2$ of $\lambda$. This is a torsion $\mco$-module.  For any finitely generated $A$-module $M$, set
	\[
	\ecoh^i_{\lambda}(M)\colonequals \tfree{\Ext^i_A(\mco,M)}
	\]
	the torsion-free quotient of the $\mco$-module  $\Ext^i_A(\mco,M)$. Here $\mco$ is viewed as an $A$-module via $\lambda$. The \emph{congruence module} of $M$ at $\lambda$ is the $\mco$-module
	\[
	\cmod{\lambda}(M)\colonequals \Coker\left(\ecoh^c_{\lambda}(M) \xra{\ \ecoh^c(\eps)\ }\ecoh^c_{\lambda}(M/\fp M)\right) 
	\]
	where $c\colonequals \height{\fp_{\lambda}}$ and $\eps \colon M\to M/\fp M$ is the natural surjection.
\end{chunk}

\begin{chunk}
	\label{ch:acat}
	Let $A$ be an $\mco$-algebra and $\lambda\colon A\to\mco$ a map of $\mco$-algebras. The following conditions are equivalent:
	\begin{enumerate}[\quad\rm(1)]
		\item
		The local ring $A_{\fp_\lambda}$ is regular.
		\item
		The rank of the $\mco$-module $\fp_\lambda/\fp^2_\lambda$ is $\height \fp_\lambda$.
		\item
		The $\mco$-module $\cmod{\lambda}(A)$ is torsion.
		\item
		The $\mco$-module $\cmod{\lambda}(M)$ is torsion for each finitely generated $A$-module $M$.
	\end{enumerate}
	Moreover, when these conditions hold the $\mco$-module $\cmod{\lambda}(A)$ is cyclic. 
	
	Condition (2) is that the embedding dimension of the ring $A_{\fp_\lambda}$ is equal to its Krull dimension, so the equivalence of (1) and (2) is one definition of regularity; see \cite[Definition~2.2.1]{Bruns/Herzog:1998}. For the other equivalences see
	\cite[Theorem 2.5 and Lemma~2.6]{Iyengar/Khare/Manning:2022a}.
	
	The pairs $(A,\lambda)$ satisfying the equivalent conditions above are the objects of a category $\acat$. A morphism $\vf\colon (A,\lambda)\to (A',\lambda')$ in $\acat$ is a map of $\mco$-algebras $\vf\colon A\to A'$ over $\mco$; that is to say, with $\lambda'\circ \vf = \lambda$. For a natural number $c$, the subcategory $\acat(c)$ of $\acat$ consists of pairs $(A,\lambda)$ such that $\height \fp_\lambda=c$.
\end{chunk}

\begin{chunk}
	\label{ch:cmod-properties}
Fix a pair $(A,\lambda)$  in $\acat$ and a finitely generated $A$-module $M$.  Since $A$ is regular at $\fp_\lambda$, and in particular a domain, the $A_{\fp_\lambda}$-module $M_{\fp_\lambda}$ has a rank. The \emph{Wiles defect} of $M$ at $\lambda$ is the integer
	\[
	\delta_\lambda(M)\colonequals \rank_{A_{\fp_\lambda}}(M_{\fp_\lambda}) \cdot \length_{\mco}\con{\lambda}(A) -  \length_{\mco}\cmod{\lambda}(M)\,.
	\]
	In particular the Wiles defect of $A$ at $\lambda$ is 
	\[
	\length_{\mco}\con{\lambda}(A) - \length_{\mco}\cmod{\lambda}(A)\,.
	\]
	We refer to \cite[Introduction]{Iyengar/Khare/Manning:2022a} for a discussion on precedents to this definition. Here are some salient properties of the Wiles defect.
	
	\begin{enumerate}[\quad\rm(1)]
		\item
		\label{it:defect-ci}
		If $A$ is complete intersection, then $\delta_\lambda(A)=0$; the converse holds when in addition $\depth A\ge c+1$.
		\item
		\label{it:defect-positive}
		If $\depth_AM\ge c+1$, then $\delta_{\lambda}(M) \ge 0$.
		\item
		\label{it:defect-invariance}
		If $\vf\colon A'\to A$ is a surjective map in $\acat(c)$ and $\depth_AM\ge c$, then
		\[
		\delta_{\lambda\vf}(M) = \delta_{\lambda}(M)\,.
		\]
		We refer to this property as the invariance of domain for congruence modules.
		\item
		\label{it:defect-freeness}
		Assume $A$ is Gorenstein and $M$ is maximal Cohen--Macaulay.  If $\delta_{\lambda}(M) = \mu \cdot \delta_{\lambda}(A)$, for $\mu\colonequals \rank_{A_{\fp_\lambda}}(M_{\fp_\lambda})$, then there is an isomorphism of $A$-modules
		\[
		M\cong A^\mu \oplus W \quad\text{and $W_{\fp_\lambda}=0$.}
		\]
	\end{enumerate}
	These results are established in \cite{Iyengar/Khare/Manning:2022a}: For (1) and (2), see Theorem A; for (3), see Theorem E, and for (4) see Theorem B.
\end{chunk}

\section{The setup}
\label{se:setup}

Fix a finite set $T$. In the number theoretic applications, $T$ will be a finite set of primes in the ring of integers $\mco_F$ of a number field $F$. However, for ease of notation, in this section we take $T\colonequals \{1,\cdots,n\}$ for some integer $n\ge 1$.

\begin{chunk}
	\label{ch:rings}
	Fix an integer $g\ge 1$ and the $\mco$-algebra
	\[
	A\colonequals \frac{\mco\pos{x_1,\dots,x_n,y_1,\dots,y_n,t_1,\dots,t_g}}{(x_1y_1,\dots, x_ny_n)}\,.
	\]
	For each subset $\Sigma\subseteq T$, set
	\[
	A_{\Sigma}\colonequals A/(x_i\mid i\notin\Sigma)\,.
	\]
	Evidently each $A_{\Sigma}$ is a reduced complete intersection, of dimension $n+g$.  Moreover the ring $A_{\varnothing}\cong A/(\bs x)$ is regular. Given subsets $\Sigma'\subseteq \Sigma$ there is an induced surjection 
	\[
	A_{\Sigma}\twoheadrightarrow A_{\Sigma'}\,,
	\]
	of $\mco$-algebras. In particular, the family $\{A_{\Sigma}\}$, as $\Sigma$ various over subsets of $T$, has an initial object, $A_{T}=A$, and final object $A_{\varnothing}$.
	
	For each $\Sigma\subseteq T$, consider the ideal 
	\[
	I_{\Sigma}\colonequals A(x_i\mid i\notin\Sigma) + A(y_j\mid j\in \Sigma)\,.
	\]
	Then the Zariski closed subsets $V(I_{\Sigma})$, as $\Sigma$ varies over the subsets of $T$, are the irreducible components of $\Spec A$. Moreover
	\[
	\Spec A_{\Sigma} = \bigcup_{\Sigma'\subseteq\Sigma} V(I_{\Sigma'})
	\]
	viewed as subsets of $\Spec A$. Thus, the irreducible components of $\Spec A_{\Sigma}$ are a subset of the irreducible components of $\Spec A$.
\end{chunk}

\begin{chunk}
	\label{ch:o-points}
	We focus on the $\mco$-points in  $\Spec A$, namely, the subset
	\[
	\mcv \colonequals \{(a_1,\dots,a_n,b_1,\dots,b_n,c_1,\dots,c_g)\in \varpi(\mco)^{2n+g} \mid a_ib_i = 0\quad \text{for $1\le i\le n$}. \}
	\]
	Since $\mco$ is a domain $a_ib_i=0$ if and only if  $a_i=0$ or $b_i=0$. Each point $\bs v \colonequals (\bs a, \bs b, \bs c)$ in $\mcv$ corresponds to the prime ideal 
	\[
	\fp_{\bs v}\colonequals (x_i-a_i,y_j-b_j,t_k-c_k\mid 1\le i,j\le n \text{ and } 1\le k\le g)\,.
	\]
	This is the kernel of the map of local $\mco$-algebras
	\[
	\lambda_{\bs v}\colon A\lra \mco \quad\text{where $\lambda_{\bs v}(x_i)=a_i, \lambda_{\bs v}(y_j)=b_i, \text{ and } \lambda_{\bs v}(t_k)=c_k$.}
	\]
	
	For  $\Sigma\subseteq T$ consider the following subsets of $\mcv$:
	\begin{align*}
		\mcv_\Sigma & \colonequals \{(\bs a,\bs b,\bs c)\in \mcv\mid \text{$a_i = 0$ for $i\not\in\Sigma$}\} \\
		\mcz_\Sigma & \colonequals \{(\bs a,\bs b,\bs c)\in \mcv\mid \text{$a_i = 0$ for $i\not\in\Sigma$ and $b_j=0$ for $j\in\Sigma$.}\} 
	\end{align*}
	Thus $\mcv_\Sigma$ are the $\mco$-valued points in $\Spec A_{\Sigma}$ viewed as a subset $\Spec A$. Furthermore, $\mcz_\Sigma$, as $\Sigma$ varies over the subsets of $T$, are the $\mco$-valued point of the irreducible components of $\Spec A$; see \ref{ch:rings}. Set
	\begin{equation}
		\label{eq:interior-component}
		\begin{aligned}
			{\mcz}^{\circ}_{\Sigma} 
			&\colonequals \mcz_{\Sigma}\setminus \bigcup_{\Sigma'\ne \Sigma} \mcz_{\Sigma} \\
			&=  \{(\bs a,\bs b,\bs c)\in \mcv\mid \text{$a_i \ne 0$ for  $i\in\Sigma$ and $b_j \ne 0$ for  $j\notin\Sigma$.}\} 
		\end{aligned}
	\end{equation}
	These are the $\mco$-valued points that lie in a single irreducible component of $\Spec A$. Observe that in ${\mcz}^{\circ}_{\Sigma}$ there are points $\bs v$ where a given subset of the components $\bs a,\bs b$ is fixed, and the other components are of any specified order. This fact is the key in the proof of Theorem~\ref{th:only}.
	
	Each $\bs v$ in $\mcz^{\circ}_{\Sigma} $ is a regular point of $A$, in that the local ring  $A_{\fp_{\bs v}}$ is regular. Set
	\begin{equation}
		\label{eq:interior-points}
		{\mcv}^{\circ}_{\Sigma} \colonequals \bigsqcup_{\Sigma'\subseteq \Sigma} {\mcz}^{\circ}_{\Sigma'}\,; 
	\end{equation}
	these are the $\mco$-valued points on $\Spec A_{\Sigma}$ that lie in a single component. Observe that the surjection $A\to A_{\Sigma}$ is an isomorphism at each $\bs v$ in $\mcv^{\circ}_{\Sigma}$, in that localization at $\bs v$ is an isomorphism
	\[
	A_{\fp_{\bs v}} \xra{\ \cong\ } (A_\Sigma)_{\fp_{\bs v}}\,.
	\]
	Thus $\bs v$ is regular also as a point in $\Spec A_{\Sigma}$; said otherwise,  $(A_\Sigma,\lambda_{\bs v})$ is in $\acat$.
\end{chunk}

\begin{chunk}
	\label{ch:cotangent-module}
	We compute the cotangent modules of the rings $A_\Sigma$ at various points in $\mcv^{\circ}_{\Sigma}$. To ease up notation, we set
	\[
	\con {\bs v}(-) \colonequals \con{\lambda_{\bs v}}(-)\,.
	\]
	Here is a simple computation. 
	
	\begin{lemma}
		\label{le:cotangent-Rsigma}
		For $\Sigma'\subseteq \Sigma$ and $\bs v$ is in $\mcz^{\circ}_{\Sigma'}$ one gets
		\[
		\con{\bs v}(A_{\Sigma}) = \bigoplus_{i\in\Sigma'} \frac{\mco[y_i]}{a_i[y_i]} \bigoplus_{j\in\Sigma\setminus\Sigma'} \frac{\mco[x_j]}{b_j[x_j]}\,,
		\]
		and hence $\length_{\mco} \con {\bs v}(A_{\Sigma}) = \sum_{i\in\Sigma'} \ord(a_i) + \sum_{j\in\Sigma\setminus\Sigma'} \ord(b_j)$.
	\end{lemma}
	
	\begin{proof}
		Without loss of generality we can assume $\Sigma=T$. For $\bs v \colonequals (\bs a, \bs b, \bs c)$  in $\mcv_{\Sigma}$ one has
		\[
		x_iy_i = b_i(x_i-a_i) + a_i(y_i-b_i) + (x_i-a_i)(y_i-b_i)
		\]
		so it is immediate from the description of $\fp_{\bs v}$, see \ref{ch:o-points}, that the cotangent module $\fp_{\bs v}/\fp_{\bs v}^2$ is generated as an $\mco$-module by classes $[x_i-a_i], [y_i-b_i]$, for $i\in T$, and $[t_k-c_k]$ for $1\le k\le g$, subject to the relations
		\[
		b_i[x_i-a_i]+a_i[y_i-b_i] = 0\quad\text{for $i\in T$}\,.
		\]
		That is to say, there is an isomorphism of $\mco$-modules
		\[
		\fp_{\bs v}/\fp_{\bs v}^2 = \bigoplus_{i\in T} \frac{\mco[x_i-a_i]\oplus \mco[y_i-b_i]}{b_i[x_i-a_i]+a_i[y_i-b_i]} \bigoplus_{k=1}^{g} \mco[t_k-c_k]
		\]
		Thus when $\bs v$ is in $\mcz^{\circ}_{\Sigma'}$, it is clear from \eqref{eq:interior-component} that
		\[
		\con{\bs v}(A) = \bigoplus_{i\in\Sigma'}\frac{\mco[y_i]}{a_i[y_i]} \bigoplus_{j\not\in\Sigma'}\frac{\mco[x_j]}{b_j[x_j]} 
		\]
		This is the desired result.
	\end{proof}
	
	We also need to compute the cotangent module of the ring
	\[
	B\colonequals A_{\Sigma}/I \quad \text{where}\quad I\colonequals \bigcap_{\Sigma'\subsetneq \Sigma} I_{\Sigma'}\,.
	\]
	This ring is not complete intersections when $|\Sigma|\ge 2$.
	
	\begin{lemma}
		\label{le:cotangent-S}
		Fix a subset $\Sigma\subseteq T$ and let $B$ be the quotient ring defined above. Fix $s\in \Sigma$ and set $\Sigma'\colonequals \Sigma\setminus\{s\}$. For any $\bs v\in \mcz^{\circ}_{\Sigma'}$ one has
		\[
		\length_{\mco} \con {\bs v}(B) = \sum_{i\in\Sigma'} \ord(a_i) + \min\{\ord(b_s), \sum_{i\in\Sigma'} \ord(a_i)\}\,.
		\]
	\end{lemma}
	
	\begin{proof}
		Without loss of generality, we can assume $\Sigma=T$ and $s=1$. Then
		\[
		B=A/I \quad\text{where $I = (x_1\cdots x_n)$.}
		\]
		For $\bs v = (\bs a,\bs b, \bs c)$ in $\mcz^{\circ}_{\Sigma'}$ one has $a_1=0$ and $a_i\ne 0$ for $i\ge 2$. Then 
		\[
		x_1\cdots x_n = x_ 1(x_2-a_2+a_2)\cdots (x_n-a_n+a_n) =  (a_2\dots a_n) x_1 + z
		\]
		where $z$ is a sum of monomials quadratic or higher in the $(x_i-a_i)$. It follows that $\con{\bs v}(B)$ is the quotient of $\con{\bs v}(A)$ by the term $(a_2\dots a_n)[x_1]$. Thus the stated inequality is immediate from the description of $\con {\bs v}(A)$ in Lemma~\ref{le:cotangent-Rsigma}.
	\end{proof}
\end{chunk}

\begin{chunk}
	\label{ch:mcm}
	We keep the notation from \ref{ch:rings} and \ref{ch:o-points}. Let $M_{\Sigma}$ be a maximal Cohen--Macaulay $A_\Sigma$-module; it is also maximal Cohen--Macaulay as an $A$-module.
	
	For each point $\bs v$ in $\mcv^{\circ}_{\Sigma}$, described in \eqref{eq:interior-points},   the ring $A_{\Sigma}$ is regular at $\bs v$, so $M_{\Sigma}$ is free at $\bs v$, that is to say, $(M_\Sigma)_{\fp_{\bs v}}$ is free over $(A_\Sigma)_{\fp_{\bs v}}\cong A_{\fp_{\bs v}}$; this is by the Auslander-Buchsbaum equality; see \cite[Theorem~1.3.3]{Bruns/Herzog:1998}. Moreover, the rank of this free module is the same at any $\bs v\in \mcz^{\circ}_{\Sigma'}$, for any $\Sigma'\subseteq\Sigma$, for these lie in the same irreducible component of $\Spec A_{\Sigma}$. We denote this number $\rank_{\Sigma'}(M_{\Sigma})$; thus
	\[
	\rank_{\Sigma'} (M_{\Sigma}) = \rank_{A_{\fp_{\bs v}}}(M_\Sigma)_{\fp_{\bs v}}\,.
	\]
	We extend this to all $\Sigma'\subseteq T$ by setting $\rank_{\Sigma'}(M_{\Sigma})=0$ when $\Sigma'\not\subseteq \Sigma$, because $(M_\Sigma)_{\fp_{\bs v}}=0$
	for $\bs v\in \mcz^{\circ}_{\Sigma'}$.
\end{chunk}

\begin{chunk}
	\label{ch:modules}
	We fix  a family of modules $M_{\Sigma}$, where $M_{\Sigma}$ is an $A_{\Sigma}$-module, equipped with $A_{\Sigma}$-linear surjections
	\[
	\pi_{\Sigma,\Sigma'}\colon M_{\Sigma}\twoheadrightarrow M_{\Sigma'}\,,
	\]
	whenever $\Sigma'\subseteq \Sigma$, satisfying the following properties:
	\begin{enumerate}[\quad\rm(1)]
		\item
		$M_{\Sigma}$ is a self-dual, maximal Cohen--Macaulay, $A_{\Sigma}$-module;
		\item
		$\pi_{\Sigma,\Sigma'}$ is an isomorphism at $\fp_{\bs v}$ for each $\bs v$ in $\mcv^{\circ}_{\Sigma}$.
		\item
		For any integer $s\not\in\Sigma'$, and $\Sigma\colonequals \Sigma'\cup \{s\}$, the composition
		\[
		M_{\Sigma'} \cong (M_{\Sigma'})^\vee \xra{\ (\pi_{\Sigma,\Sigma'})^\vee \ } (M_{\Sigma})^\vee \cong M_{\Sigma}
		\xra{\ \pi_{\Sigma,\Sigma'} \ } M_{\Sigma'}\,,
		\]
		is multiplication by $y_s$.
	\end{enumerate}
	
	Following the discussion in \ref{ch:mcm}, for $\Sigma\subseteq T$ set
	\[
	\mu_{\Sigma} \colonequals \rank_{\Sigma}M_{T}\,.
	\]
	Given the isomorphism in condition (2) above, for subsets $\Sigma,\Sigma'\subseteq T$ one has
	\[
	\rank_{\Sigma'}M_{\Sigma} =
	\begin{cases}
		\mu_{\Sigma'} & \text{when $\Sigma'\subseteq \Sigma$}\\
		0 &\text{otherwise}
	\end{cases}
	\]

	To ease up notation, we set
	\[
	\cmod {\bs v}(M) \colonequals \cmod {\lambda_{\bs v}}(M)\,.
	\]
	Here is a computation of congruence modules.
	
	\begin{lemma}
		\label{le:cmod}
		Fix $\Sigma'\subseteq \Sigma\subseteq T$ and $\bs v \colonequals (\bs a, \bs b, \bs c)$  in $\mcz^{\circ}_{\Sigma'}$ one has
		\[
		\length_{\mco} \cmod {\bs v}(M_{\Sigma}) = \length_{\mco} \cmod {\bs v}(M_{\Sigma'}) 
		+ \mu_{\Sigma'}(\sum_{i\in\Sigma\setminus \Sigma'} \ord(b_i))\,.
		\]
	\end{lemma}
	
	\begin{proof}
		We verify that by an induction on the cardinality of $\Sigma\setminus \Sigma'$. The base case is when $\Sigma=\Sigma'$, and then the stated equality is clear.
		
		For the induction step it suffices to note that if $\Sigma = \Sigma' \cup\{s\}$ for $s\notin  \Sigma'$, then
		\[
		\length_{\mco} \cmod {\bs v}(M_{\Sigma}) = \length_{\mco} \cmod {\bs v}(M_{\Sigma'}) + \rank_{\bs v}(M_\Sigma) \ord(b_s)\,.
		\]
		This is immediate by property \ref{ch:modules}(3) and \cite[Proposition~4.4]{Iyengar/Khare/Manning:2022a}. 
	\end{proof}
	
	\begin{proposition}
		\label{pr:free-summand}
		For each $\Sigma\subseteq T$ the $A_\Sigma$-module $M_{\Sigma}$ has a free summand of rank $\mu_{\varnothing}$. In particular, one gets an inequality $\mu_{\Sigma}\ge \mu_{\varnothing}$.
	\end{proposition}
	
	\begin{proof}
		Pick a $\bs v$ in $\mcz^\circ_{\varnothing}$. Since the local ring $A_{\varnothing}$ is regular and $M_{\varnothing}$ is maximal Cohen--Macaulay, it is free, and of rank equal to $\mu_{\varnothing}$. Thus $\cmod {\bs v}(M_\varnothing)=0$ and hence from Lemma~\ref{le:cmod} applied with $\Sigma'=\varnothing$ one gets the first equality below:
		\begin{align*}
			\length_{\mco} \cmod {\bs v}(M_{\Sigma})
			&=\mu_{\varnothing}\cdot \big(\sum_{i\in\Sigma} \ord(b_i)\big) \\
			&=\mu_{\varnothing} \cdot \length_{\mco} \con {\bs v}(A_\Sigma)\,.
		\end{align*}
		The second equality is by Lemma~\ref{le:cmod}, applied with $\Sigma'=\varnothing$. Since $A_\Sigma$ is complete intersection, and hence Gorenstein, it remains to apply \ref{ch:cmod-properties}(4).
	\end{proof}

\end{chunk}

\section{A criterion for freeness}
\label{se:freeness} 

The theorem below builds on and strengthens Proposition~\ref{pr:free-summand}.

\begin{theorem}
	\label{th:only}
	Let $\{A_\Sigma\}$ and $\{M_\Sigma\}$ be the families of rings and modules described in \ref{ch:rings} and \ref{ch:modules}, respectively. If $\mu_{T}\le \mu_{\varnothing}$, then  $M_{\Sigma}\cong A_{\Sigma}^{\mu_{\varnothing}}$ for each $\Sigma\subseteq T$.
\end{theorem}

\begin{proof}
	By Proposition~\ref{pr:free-summand},  for each $\Sigma\subseteq T$ there is an $A_{\Sigma}$-module $W_{\Sigma}$ and an isomorphism of $A_\Sigma$-modules
	\begin{equation}
		\label{eq:only-summands}
		M_{\Sigma}\cong A_{\Sigma}^{\mu_\varnothing} \oplus W_{\Sigma}\,.
	\end{equation}
	We prove that $W_{\Sigma}$ is zero. To that end it suffices to prove $\mu_{\Sigma}\le \mu_{\varnothing}$ for each $\Sigma\subseteq T$. 
	
Indeed, then $\rank_{\bs v}(M_\Sigma) \le \mu_{\varnothing}$ for $\bs v\in \mcz^\circ_{\Sigma}$, as described in \eqref{eq:interior-points}. Thus the isomorphism in \eqref{eq:only-summands} implies that $\rank_{\bs v}(W_{\Sigma})=0$ for each such $\bs v$. This implies that the support of $W_{\Sigma}$ does not contain any component of $\Spec A_\Sigma$. Since $W_{\Sigma}$ is a maximal Cohen--Macaulay $A_\Sigma$-module, we conclude that it is zero, as desired.
	
	It thus remains to verify that 
	\[
	\mu_{\Sigma}\le \mu_{\varnothing}\quad\text{for each $\Sigma\subseteq T$.}
	\]
We verify this by computing the congruence module of $W_\Sigma$ at suitable points $\bs v$ in $\mcv^\circ_{\Sigma}$, and proving that, if the desired inequality does not hold, then there are points $\bs v$ at which the Wiles defect of $W_{\Sigma}$ would be negative, violating~\ref{ch:cmod-properties}\eqref{it:defect-positive}.
	
Since the inequality above holds when $\Sigma=T$, it suffices to prove that if the inequality holds for a given $\Sigma$, it holds also for $\Sigma' \colonequals \Sigma\setminus \{s\}$ for any $s\in\Sigma$.  
	
	With $\Sigma$ and $\Sigma'$ as above, fix $\bs v$ in $\mcz^\circ_{\Sigma'}$. The isomorphisms in \eqref{eq:only-summands}, applied to $\Sigma$ and $\Sigma'$, yield the first equality and the third inequality, respectively, below:
	\begin{equation}
		\label{eq:only-estimates}
		\begin{aligned}
			\length_{\mco}\cmod {\bs v}(W_{\Sigma}) 
			&= \length_{\mco}\cmod {\bs v}(M_{\Sigma}) - \mu_{\varnothing}\cdot \length_{\mco}\cmod {\bs v}(A_{\Sigma}) \\
			&= \length_{\mco}\cmod {\bs v}(M_{\Sigma'}) + \mu_{\Sigma'}\cdot \ord(b_s)
			- \mu_{\varnothing}\cdot \length_{\mco}\cmod {\bs v}(A_{\Sigma}) \\
			&\ge \mu_{\varnothing}\cdot \length_{\mco}\cmod {\bs v}(A_{\Sigma'})	+ \mu_{\Sigma'}\cdot \ord(b_s)
			- \mu_{\varnothing}\cdot \length_{\mco}\cmod {\bs v}(A_{\Sigma}) \\ 
			&=\mu_{\varnothing}\cdot \left[\length_{\mco}\cmod {\bs v}(A_{\Sigma'}) -  \length_{\mco}\cmod {\bs v}(A_{\Sigma})\right]
			+ \mu_{\Sigma'}\cdot \ord(b_s) \\
			&= (\mu_{\Sigma'} -\mu_{\varnothing})\cdot \ord(b_s)\,.
		\end{aligned}
	\end{equation}
	The second equality is by Lemma~\ref{le:cmod} while the last one is by Lemma~\ref{le:cotangent-Rsigma}.
	
	Since $\mu_\Sigma \le \mu_\varnothing$, the $A_\Sigma$-module $W_{\Sigma}$ is not supported on $\mcz^\circ_{\Sigma}$. Since $W_{\Sigma}$ is also maximal Cohen--Macaulay, it follows that the $A_\Sigma$ action on it factors through the quotient ring $B\colonequals A_{\Sigma}/I$, where
	\[
	I\colonequals \bigcap_{\Sigma''\subsetneq \Sigma} I''_{\Sigma}
	\]
	considered in Lemma~\ref{le:cotangent-S}. Observe that the surjection $A_{\Sigma}\to A_{\Sigma'}$ factors through the  map $A_{\Sigma}\to B$. In particular, the latter is also an isomorphism at $\bs v$, and hence the pair $(B,\lambda_{\bs v}\circ \varepsilon)$, where $\varepsilon\colon B\to A_{\Sigma'}$ is the canonical surjection, is also in $\acat$, and the results in \ref{ch:cmod-properties} apply to this pair as well.
	
	Since $W_{\Sigma}$ is maximal Cohen--Macaulay, the invariance of domain property of congruence modules~\ref{ch:cmod-properties}\eqref{it:defect-invariance} gives the equality below:
	\begin{equation}
		\label{eq:only-S}
		\begin{aligned}
			(\mu_{\Sigma'}-\mu_{\varnothing})\cdot \length\con {\bs v}(B)
			&\ge \length_{\mco}\cmod {\bs v}^B(W_{\Sigma}) \\
			&=\length_{\mco}\cmod {\bs v}(W_{\Sigma}) \\
			&\ge (\mu_{\Sigma'} -\mu_{\varnothing})\cdot \ord(b_s)\,,
		\end{aligned}
	\end{equation}
	where $\cmod {\bs v}^B(W_{\Sigma})$ is the congruence module of $W_{\Sigma}$, treated as a $B$-module. The first inequality holds because the Wiles defect of the $B$-module $W_{\Sigma}$ at $\bs v$ is non-negative; see~\ref{ch:cmod-properties}\eqref{it:defect-positive}. The last inequality is from \eqref{eq:only-estimates}.
	
	Suppose, contrary to the desired result, $\mu_{\Sigma'}> \mu_{\varnothing}$. Then \eqref{eq:only-S} implies
	\[
	\length\con {\bs v}(B) \ge \ord(b_s) \quad\text{for each $\bs v\in \mcz^\circ_{\Sigma'}$.}
	\]
	On the other hand, from Lemma~\ref{le:cotangent-S} we get that
	\[
	\length_{\mco} \con {\bs v}(B) = \sum_{i\in\Sigma} \ord(a_i) + \min\{\ord(b_s), \sum_{i\in\Sigma} \ord(a_i)\}\,.
	\]
	In particular, when  $b_s\ge \sum_{i\in\Sigma} \ord(a_i)$ combining the (in)equalities above yields
	\[
	2\sum_{i\in\Sigma} \ord(a_i)\ge \ord(b_s) \quad\text{for each $\bs v\in \mcz^\circ_{\Sigma'}$.}
	\]
	It is clear from \eqref{eq:interior-component} that there are points $\bs v\in \mcz^\circ_{\Sigma'}$ with $\ord(b_s)$ arbitrarily large and $a_i$ fixed. This violates the inequality above, yielding the contradiction we seek.
\end{proof}

\begin{chunk}
Theorem~\ref{th:only} addresses only rings of the form considered in (\ref{ch:rings}). This is primarily for ease of exposition, as this is sufficient to deal with the rings considered in \cite{Iyengar/Khare/Manning:2022a}. One could prove analogues of Theorem \ref{th:only} with the rings $A_\Sigma$ replaced by more general classes of Gorenstein rings.

Indeed, there are more general families of rings $A_\Sigma$ and modules $M_\Sigma$ for which the argument of Proposition \ref{pr:free-summand} applies. Besides Proposition \ref{pr:free-summand}, the primary input into the proof of Theorem \ref{th:only} is Lemma \ref{le:cotangent-S}, which allows one to pick points $\bs v\in \mcz^\circ_{\Sigma'}$ making $\length_{\mco} \Psi_{\bs v}(M_\Sigma)$ arbitrarily large, while keeping $\length_{\mco} \Phi_{\bs v}(B)$ bounded. But one can find such $\bs v$ in  more general cases than we have considered here (roughly, one merely needs to find a point $\bs v_0\in \mcz_{\Sigma'}\cap \mcz_\Sigma$ at which $\Spec B$ is smooth, and pick $\bs v\in \mcz_{\Sigma'}^\circ$ to be sufficiently close to $\bs v_0$).

This could potentially allow one to prove freeness of Hecke modules at more general levels than we consider here and in \cite{Iyengar/Khare/Manning:2022a}.
\end{chunk}

\section{Applications to Hecke modules}
\label{se:applications}

In this section we apply  our results from Section~\ref{se:freeness} to the number theoretic context explored in  \cite{Iyengar/Khare/Manning:2022a}.  We freely use the notation and results of that paper.

Let $F$ be a number field and assume Conjectures A, B, C and D from \cite{Iyengar/Khare/Manning:2022a} hold for $F$. Let $r_1$ and $r_2$ be the number of real and complex places of $F$ respectively, so that $r_1+2r_2=[F:\QQ]$.

Pick a prime $\ell>2$ which does not ramify in $F$, and let $E/\QQ_\ell$ be a finite extension with ring of integers $\mco$, uniformizer $\varpi\in \mco$ and residue field $k = \mco/\varpi\mco$. We  use this ring as the ring $\mco$ from \cite{Iyengar/Khare/Manning:2022a}, and the DVR from Section \ref{se:congruence}.

Let $\mcn_\es\subseteq \mco_F$ be a nonzero ideal which is relatively prime to $\ell$. Let $K_0(\mcn_\es)\subseteq \PGL_2(\bbA_F^\infty)$ be the corresponding compact open subgroup from \cite[Section 13]{Iyengar/Khare/Manning:2022a} (and recall that $K_0(\mcn_\es)$ was defined to have additional level structure at a certain auxiliary prime). Let $\fm\subseteq {\TT}(K_0(\mcn_\es))$ be a non-Eisenstein maximal ideal with residue field $k$, for which $N(\rhobar_{\fm}) = \mcn_\es$.

Let $\Sigma$ be a finite set of primes of $\mco_F$ such that for all $v\in \Sigma$:
\begin{itemize}
	\item $v\nmid \mcn_\es$.
	\item $\rhobar_{\fm}$ is unramified at $v$.
	\item $v\nmid \ell$ and $q_v \not\equiv 1\pmod{\ell}$.
	\item $\rhobar_{\fm}(\Frob_v)$ has eigenvalues $q_v\epsilon_v$ and $\epsilon_v$ for some $\epsilon_v=\pm 1$.
\end{itemize}

We consider the following hypothesis on $\rhobar_\fm$ on $\Sigma$:

\begin{hypothesis}\label{hy:integral lift}
There is a finite set $T$ of primes of $\mco_F$ such that $\Sigma\subseteq T$ and for all $v\in T$:
\begin{itemize}
	\item $v\nmid \mcn_\es$.
	\item $\rhobar_{\fm}$ is unramified at $v$.
	\item $v\nmid \ell$ and $q_v \not\equiv 1\pmod{\ell}$.
	\item $\rhobar_{\fm}(\Frob_v)$ has eigenvalues $q_v\epsilon_v$ and $\epsilon_v$ for some $\epsilon_v=\pm 1$.
\end{itemize}
and a continuous Galois representation $\rho\colon G_F\to \GL_2(\mco')$, where $\mco'$ is the ring of integers in a finite extension $E'/\QQ_\ell$, with uniformizer $\varpi'\in\mco'$, satisfying:
\begin{itemize}
	\item $\rho\equiv \rhobar_\fm\pmod{\varpi'}$;
	\item $\det \rho = \varepsilon_\ell$;
	\item For all places $v|\ell$, $\rho$ is flat at $v$;
	\item If $v$ is any place of $F$ for which $v\nmid \ell$ and $v|\mcn_\es$, then $\rho$ is minimally ramified at $v$;
	\item If $v$ is any place of $F$ with $v\nmid\ell\mcn_\es$ and $v\not\in T$ then $\rho$ is unramified at $v$.
	\item If $v\in T$, then $\rho$ is ramified at $v$.
	\item If $v \in T$ and $q_v\equiv- 1\pmod{\ell}$ then $\rho$ arises from $R_v^{{\rm uni}(\epsilon_v)}$.
\end{itemize}
\end{hypothesis}

The question of whether such geometric  liftings exist  for representations $\rhobar:G_F \to \GL_2(k)$ when $F$ is  not a totally real field is wide open. One might optimistically expect such liftings to exist, as  there  seems to be  no strong heuristic to suggest otherwise. To make our results independent of Hypothesis \ref{hy:integral lift} one could just assume that $\rhobar_\fm$ arises by reduction of a geometric characteristic 0 representation $\rho$. There are many such representations  (in the case of CM fields $F$) associated to regular algebraic cuspidal automorphic representations  of $\GL_2(\bbA_F)$ by \cite{HLTT}.

From now on assume that $\rhobar_\fm$ and $\Sigma$ satisfy Hypothesis \ref{hy:integral lift}. Let $T$ and $\rho\colon G_F\to \GL_2(\mco')$ be as in Hypothesis \ref{hy:integral lift}.

For any $\Sigma'\subseteq T$ define $\mcn_{\Sigma'}\subseteq \mco_F$, $K_0(\mcn_{\Sigma'})\subseteq\PGL_2(\bbA_F^\infty)$ and $Y_0(\mcn_{\Sigma'})$ as in \cite{Iyengar/Khare/Manning:2022a}. Let $R_{\Sigma'}$ and $\TT_{\Sigma'}$ be associated the global deformation ring and Hecke algebra, again defined as in \cite{Iyengar/Khare/Manning:2022a}.

By expanding $E$ if necessary, we may assume that all augmentations $\lambda\colon \TT_{\Sigma'}\to \overline{\QQ}_\ell$, for all $\Sigma'\subseteq T$, have image equal to $\mco$, and so in particular $E'=E$ and $\mco'=\mco$.

Now assume that $\rhobar_\fm|_{G_{F(\zeta_\ell)}}$ is absolutely irreducible. The patching argument from \cite[Section 14]{Iyengar/Khare/Manning:2022a} produces, for each $\Sigma'\subseteq T$, a complete local noetherian $\mco$-algebra $R_{\Sigma',\infty}$ and a maximal Cohen--Macaulay $R_{\Sigma',\infty}$-module $M_{\Sigma',\infty}$. Moreover there exists a power series ring $S_\infty$ with maps $S_\infty\to R_{\Sigma',\infty}$ for which
\[
R_{\Sigma',\infty}\otimes_{S_\infty}\mco = R_{\Sigma'}\cong \TT_{\Sigma'}
\quad
\text{and}
\quad
M_{\Sigma',\infty}\otimes_{S_\infty}\mco \cong \hh_{r_1+r_2}(Y_0(\mcn_\Sigma),\mco)_{\fm_\Sigma}
\]
We apply the results of Sections \ref{se:setup} and \ref{se:freeness}, taking $A = R_{T,\infty}$, $A_{\Sigma'} = R_{\Sigma',\infty}$ and $M_{\Sigma'} = M_{\Sigma',\infty}$. The work of \cite[Section 14]{Iyengar/Khare/Manning:2022a} implies that these objects satisfy all of the properties listed in Section \ref{se:setup}. In particular, we may consider the integers $\mu_{\Sigma'} \colonequals\rank_{\Sigma'} M_{T,\infty}$.

Our result relies on the following standard generic multiplicity result:

\begin{lemma}\label{lem:generic multiplicity}
	For any $\Sigma'\subseteq T$ and augmentation $\lambda\colon \TT_{\Sigma'}\onto \mco$  we have
	\[
	\dim_E (\hh_{r_1+r_2}(Y_0(\mcn_{\Sigma'}),\mco)_{\fm_{\Sigma'}}\otimes_\lambda E)  = 2^{r_1}
	\]
\end{lemma}

\begin{proof}
Set $\fp_{\lambda}\colonequals \Ker(\lambda)$. Recall that
	\[
	\hh_{r_1+r_2}(Y_0(\mcn_{\Sigma'}),\mco)_{\fm_{\Sigma'}}\otimes_\lambda E \colonequals
	(\hh_{r_1+r_2}(Y_0(\mcn_{\Sigma'}),\mco)_{\fm_{\Sigma'}}/\fp_\lambda)\otimes_\mco E
	\]
The lemma can be deduced from  \cite[\S 3.6.2]{Harder}. 
\end{proof} 

This allows us to compute the rank $\mu_{\Sigma'}$ for all $\Sigma'$ for which an appropriate augmentation $\TT_{\Sigma'}\onto\mco$ exists:

\begin{corollary}\label{cor:mu_Sigma}
Take any $\Sigma'\subseteq T$. If there exists an augmentation $\lambda\colon \TT_{\Sigma'}\onto \mco$ such that the pullback
\[
R_{\Sigma',\infty}\onto R_{\Sigma'}\cong \TT_{\Sigma'}\xrightarrow{\lambda}\mco
\]
is equal to $\lambda_{\bs v}$ for some $\displaystyle {\bs v} \in \mcv_{\Sigma'}\sm \bigcup_{\Sigma''\subsetneq \Sigma'} \mcz_{\Sigma''}$, then $\mu_{\Sigma'} = 2^{r_1}$.
\end{corollary}
\begin{proof}
The condition that $\displaystyle {\bs v} \in \mcv_{\Sigma'}\sm \bigcup_{\Sigma''\subsetneq \Sigma'} \mcz_{\Sigma''}$ ensures that $\fp_{\bs v}$ lies in $\mcz_{\Sigma'}$ and that $\Spec R_{\Sigma',\infty}$ is regular at $\fp_{\bs v}$. It follows by the discussion in Section \ref{se:setup} that $\rank_{\fp_{\bs v}} M_{\Sigma',\infty} = \mu_{\Sigma'}$. Thus
\begin{align*}
	\mu_{\Sigma'} 
	&= \dim_E (M_{\Sigma',\infty}\otimes_{\lambda}E)\\
	&= \dim_E ((M_{\Sigma',\infty}\otimes_{R_{\Sigma',\infty}}R_{\Sigma'})\otimes_{\lambda}E)\\
	&= \dim_E ((M_{\Sigma',\infty}\otimes_{S_\infty}\mco)\otimes_{\lambda}E)\\
	&= \dim_E (\hh_{r_1+r_2}(Y_0(\mcn_{\Sigma'}),\mco)_{\fm_{\Sigma'}}\otimes_{\lambda}E)\\
	&= 2^{r_1}
\end{align*}
by Lemma \ref{lem:generic multiplicity}.
\end{proof}

In particular, the integral lift $\rho\colon G_F\to \GL_2(\mco)$ of $\rhobar_{\fm}$ from Hypothesis \ref{hy:integral lift} gives the following:

\begin{corollary}\label{cor:mu_T}
$\mu_T = 2^{r_1}$.
\end{corollary}
\begin{proof}
Recall that we have assumed $\rhobar_\fm|_{G_{F(\zeta_\ell)}}$ is absolutely irreducible. Thus standard modularity lifting results (see for instance \cite[Theorem 5.16, Theorem 9.19]{Calegari/Geraghty:2018}) imply that the representation $\rho$ is modular, and moreover is equal to $\rho_\lambda$ for some augmentation $\lambda\colon \TT_T\onto \mco$. Let the pullback
\[
R_{T,\infty}\onto R_{T}\cong \TT_{T}\xrightarrow{\lambda}\mco
\]
equal $\lambda_{\bs v}$ for some ${\bs v}\in \mcv_T$.

By the description of $R_{T,\infty}$ in \cite{Iyengar/Khare/Manning:2022a} (in particular, \cite[Proposition 12.1]{Iyengar/Khare/Manning:2022a}) the fact that $\rho$ is ramified at all primes in $T$ implies that ${\bs v}\not\in \mcz_{\Sigma'}$ for any $\Sigma'\subsetneq T$. Hence Corollary \ref{cor:mu_Sigma} gives $\mu_T = 2^{r_1}$.
\end{proof}

Now as $N(\rhobar_{\fm}) = \mcn_\es$ we have that $M_{\es,\infty} \ne 0$ and so $\mu_{\es}\ge 1$. Hence $1\le \mu_\es \le \mu_T = 2^{r_1}$ by Proposition \ref{pr:free-summand}.

If the number field $F$ is totally complex, so that $r_1=0$, we then get $\mu_\es = 1 = \mu_T$ and so Theorem \ref{th:only} implies that $M_{\Sigma,\infty}\cong R_{\Sigma,\infty}$. Applying $-\otimes_{S_\infty}\mco$, this gives that $\hh_{r_2}(Y_0(\mcn_\Sigma),\mco)_{\fm_\Sigma}\cong R_\Sigma\cong\TT_\Sigma$. We have thus proved the following:

\begin{theorem}\label{th:mult 1}
Let $F$ be a totally complex number field in which $\ell$ does not ramify, and assume Conjectures A, B, C and D from \cite{Iyengar/Khare/Manning:2022a} hold for $F$. 

Let $\mcn_\es\subseteq \mco_F$ be a nonzero ideal and let $\fm\subseteq {\TT}(K_0(\mcn_\es))$ be a non-Eisenstein maximal ideal such that $N(\rhobar_\fm) = \mcn_\es$ and $\rhobar_\fm|_{G_{F(\zeta_\ell)}}$ is absolutely irreducible. Let $\Sigma$ be a finite set of primes of $\mco_F$ such that for all $v\in \Sigma$:
\begin{itemize}
	\item $\rhobar_v$ is unramified.
	\item $v\nmid \ell$ and $q_v \not\equiv 1\pmod{\ell}$.
	\item $\rhobar_v(\Frob_v)$ has eigenvalues $q_v\epsilon_v$ and $\epsilon_v$ for some $\epsilon_v=\pm 1$.
\end{itemize}
Assume that $\rhobar_\fm$ and $\Sigma$ satisfy Hypothesis \ref{hy:integral lift}. Then $\hh_{r_2}(Y_0(\mcn_\Sigma),\mco)_{\fm_\Sigma}$ is free of rank $1$ over $\TT_\Sigma$. \qed
\end{theorem}

When $F$ is not totally complex, we get the following weaker result:

\begin{theorem}\label{th:mult 2^r1}
	Let $F$ be a number field in which $\ell$ does not ramify, and assume Conjectures A, B, C and D from \cite{Iyengar/Khare/Manning:2022a} hold for $F$. 
	
	Let $\mcn_\es\subseteq \mco_F$ be a nonzero ideal and let $\fm\subseteq {\TT}(K_0(\mcn_\es))$ be a non-Eisenstein maximal ideal such that $N(\rhobar_\fm) = \mcn_\es$ and $\rhobar_\fm|_{G_{F(\zeta_\ell)}}$ is absolutely irreducible. Let $\Sigma$ be a finite set of primes of $\mco_F$ such that for all $v\in \Sigma$:
	\begin{itemize}
		\item $\rhobar_v$ is unramified.
		\item $v\nmid \ell$ and $q_v \not\equiv 1\pmod{\ell}$.
		\item $\rhobar_v(\Frob_v)$ has eigenvalues $q_v\epsilon_v$ and $\epsilon_v$ for some $\epsilon_v=\pm 1$.
	\end{itemize}
	Assume that $\rhobar_\fm$ and $\Sigma$ satisfy Hypothesis \ref{hy:integral lift} and that $\hh_{r_1+r_2}(Y_0(\mcn_\es),E)_{\fm_\Sigma}\ne 0$.
	
	Then $\hh_{r_1+r_2}(Y_0(\mcn_\Sigma),\mco)_{\fm_\Sigma}$ is free of rank $2^{r_1}$ over $\TT_\Sigma$.
\end{theorem}

\begin{proof}
The hypothesis that $\hh_{r_1+r_2}(Y_0(\mcn_\es),E)_{\fm_\Sigma}\ne 0$ implies there exists an augmentation $\lambda'\colon \TT_\es\onto \mco$. Again, this pulls back to an augmentation
\[
\lambda_{\bs v'}\colon R_{\es,\infty}\onto R_\es\cong \TT_\es\xrightarrow{\lambda'}\mco
\]
for some ${\bs v'} \in \mcv_\es$. Vacuously we have ${\bs v'} \not\in \mcz_{\Sigma''}$ for $\Sigma''\subsetneq\es$ and so Corollary \ref{cor:mu_Sigma} gives $\mu_\es = 2^{r_1} = \mu_T$.

The claim now follows from the previous argument.
\end{proof}

\section*{Acknowledgements}
This work is partly supported by National Science Foundation grants DMS-200985 (SBI) and DMS-2200390 (CBK), and  by a Simons Fellowship (CBK). 
The second author thanks the Tata Institute for Fundamental Research in Mumbai, and the third author thanks the Max Planck Institute for Mathematics in Bonn, for support and hospitality. 

\bibliographystyle{amsplain}


\end{document}